\documentclass[12pt]{amsart}
\usepackage{amsmath,amssymb,amsfonts,amsthm,amscd,indentfirst}
\usepackage{amsmath,amssymb,amsfonts,amsthm,amscd}
\usepackage{indentfirst}
\usepackage{hyperref}

\numberwithin{equation}{section}

\newtheorem{prop}{Proposition}[section]
\newtheorem{theo}[prop]{Theorem}
\newtheorem{lemm}[prop]{Lemma}

\newtheorem{rema}[prop]{Remark}

\newtheorem{defi}[prop]{Definition}

\def\and{\quad{\rm and}\quad}

\def\<{\langle}
\def\>{\rangle}

\usepackage{graphicx}

\begin{document}
\title[On the Dirichlet problem for Lagrangian phase equation]{On the Dirichlet problem for Lagrangian phase equation with critical and supercritical phase}
\author[Siyuan Lu]{Siyuan Lu}
\address{Department of Mathematics and Statistics, McMaster University, 
1280 Main Street West, Hamilton, ON, L8S 4K1, Canada.}
\email{siyuan.lu@mcmaster.ca}
\thanks{The author was supported in part by NSERC Discovery Grant.}

\begin{abstract}
In this paper, we solve the Dirichlet problem for Lagrangian phase equation with critical and supercritical phase. A crucial ingredient is the interior $C^2$ estimate. Our result is sharp in the sense that there exist singular solutions in the subcritical phase case. 
\end{abstract}

\maketitle

\section{Introduction}

In this paper, we study the Dirichlet problem for the following fully nonlinear equation which we call Lagrangian phase equation 
\begin{align}\label{SL}
F(D^2u):=\sum_{i=1}^n \arctan(\lambda_i)=\theta(x),
\end{align}
where $\lambda_i$'s are the eigenvalues of $D^2u$.

When the phase $\theta(x)=\Theta$ is a constant, then equation (\ref{SL}) is the special Lagrangian equation
\begin{align}\label{SL-original}
F(D^2u)=\sum_{i=1}^n \arctan(\lambda_i)=\Theta.
\end{align} 

Equations (\ref{SL}) and (\ref{SL-original}) originate in the study of calibrated geometry by Harvey and Lawson \cite{HL}. Let $(x,Du(x))\subset \mathbb{R}^n\times\mathbb{R}^n$ be a Lagrangian submanifold, then the mean curvature vector is given by $H=J\nabla_g\theta$. The Lagrangian submanifold is called special if and only if it is a (volume minimizing) minimal submanifold. In other words, $(x,Du(x))$ is a special Lagrangian submanifold if and only if it satisfies equation (\ref{SL-original}). Thus equation (\ref{SL-original}) can be viewed as the potential equation for the special Lagrangian submanifold. Similarly, equation (\ref{SL}) can be viewed as the potential equation for the Lagrangian submanifold with mean curvature vector $H$. 

\medskip

The range of phase plays an important role in the study of equations (\ref{SL}) and (\ref{SL-original}). It was proved by Yuan \cite{Y06} that the level set of Lagrangian phase operator is convex only if $|\Theta|\geq \frac{(n-2)\pi}{2}$. Consequently, he introduced the following notion in \cite{Y06}.

\begin{defi}
The phase is called critical if $|\Theta|=\frac{(n-2)\pi}{2}$, supercritical if $|\Theta|>\frac{(n-2)\pi}{2}$ and subcritical if $|\Theta|<\frac{(n-2)\pi}{2}$.
\end{defi}

It turns out that this notion is crucial to the solvability of the Dirichlet problem for equation (\ref{SL}). For the supercritical phase case $|\theta(x)|>\frac{(n-2)\pi}{2}$, it was solved recently by Collins, Picard and Wu \cite{CPW}, see also \cite{B2,CNS}. In the subcritical phase case $|\theta(x)|< \frac{(n-2)\pi}{2}$, singular solutions were constructed by Nadirashvili and Vlăduţ \cite{NV} and Wang and Yuan \cite{WdY13}. 

From analysis point of view, it is desirable to study the borderline case. In this paper, we solve the Dirichlet problem for the critical and supercritical phase case $|\theta(x)|\geq \frac{(n-2)\pi}{2}$. The following is our main theorem.
\begin{theo}\label{theorem}
Let $\Omega$ be a bounded, strictly convex domain in $\mathbb{R}^n$ for $n\geq 3$. Let $\varphi\in C^{1,1}(\partial\Omega)$ and let $\theta\in C^{1,1}(\overline{\Omega})$ with $|\theta(x)|\in \big[ \frac{(n-2)\pi}{2}, \frac{n\pi}{2}\big)$. Then the Dirichlet problem
\begin{align}\label{Dirichlet Problem}
\begin{cases}
F(D^2u)=\theta(x),\quad \textit{in }\Omega,\\
u=\varphi,  \quad\textit{on } \partial\Omega,
\end{cases}
\end{align}
admits a unique solution $u\in C^{3,\alpha}(\Omega)\cap C^{0,1}(\overline{\Omega})$. Moreover, if $\theta$ is smooth, then $u$ is also smooth in $\Omega$.
\end{theo}

\begin{rema}
Our result is sharp in view of singular solutions in the subcritical phase case $|\theta(x)|< \frac{(n-2)\pi}{2}$. This completes the picture for the solvability of the Dirichlet problem for Lagrangian phase equation.
\end{rema}

The Dirichlet problem for a broad class of fully nonlinear equations was studied in a seminal paper by Caffarelli, Nirenberg and Spruck \cite{CNS}, generalizing their previous work on Monge-Amp\`ere equations \cite{CNS84}. Since then, the Dirichlet problem for fully nonlinear equations under various structure conditions has been studied, see \cite{T} and references therein. An essential requirement for those equations is the concavity of the operator $F$, which is heavily used in the regularity estimates: gradient estimate, $C^2$ estimate as well as $C^{2,\alpha}$ estimate. Without concavity condition, even for uniformly elliptic operator $F$, there exist singular solutions constructed by Nadirashvili and Vlăduţ \cite{NV07}.

For the Lagrangian phase operator, a key observation in \cite{CPW} is that the operator $G=-e^{-AF}$ is in fact concave for the supercritical phase case $\theta(x)>\frac{(n-2)\pi}{2}$. This allows them to solve the Dirichlet problem for equation (\ref{SL}) in the supercritical phase case. In contrast, it is unclear whether there is a hidden concavity in the critical and supercritical phase case. Consequently, the a priori estimates are extremely delicate. 

A crucial ingredient in our proof is the following interior $C^2$ estimate, which is of independent interest.
\begin{theo}\label{thm-C^2-intro}
Let $B_2$ be a ball of radius $2$ in $\mathbb{R}^n$ for $n\geq 3$. Let $\theta\in C^{1,1}(B_2)$ with $|\theta(x)|\in \big[ \frac{(n-2)\pi}{2}, \frac{n\pi}{2}\big)$ and let $u\in C^4(B_2)$ be a solution of equation (\ref{SL}). Then we have
\begin{align*}
|D^2u(0)|\leq C,
\end{align*}
where $C$ depends only on $n,\|u\|_{C^{0,1}(\overline{B}_1)}$ and $\|\theta\|_{C^{1,1}(\overline{B}_1)}$.
\end{theo}

Another important ingredient is the gradient estimate. Our interior to boundary gradient estimate works for all domains and for all phase functions including subcritical phase case, see Lemma \ref{C^1}. It was kindly pointed out by the anonymous referee that the following gradient estimate is well known to experts. We would like to thank the referee for his/her generosity for sharing this theorem as well as its proof.
\begin{theo}\label{coro}
Let $\Omega$ be a bounded, strictly convex domain in $\mathbb{R}^n$ for $n\geq 3$. Let $\varphi\in C^{1,1}(\partial\Omega)$ and let $\Theta\in \left(-\frac{n\pi}{2},\frac{n\pi}{2}\right)$ be any constant. Let $u$ be a viscosity solution of the Dirichlet problem
\begin{align*}
\begin{cases}
F(D^2u)=\Theta,\quad \textit{in }\Omega,\\
u=\varphi,  \quad\textit{on } \partial\Omega.
\end{cases}
\end{align*}
Then we have
\begin{align*}
\|u\|_{C^{0,1}(\overline{\Omega})}\leq C,
\end{align*}
where $C$ depends only on $n$, $\Omega$ and $\|\varphi\|_{C^{1,1}(\partial\Omega)}$.
\end{theo}

\begin{rema}
Recall that the singular solutions $u_m$ in \cite{WdY13} satisfy $u_m\in C^{1,\frac{1}{2m-1}}$, but $u_m\notin C^{1,\delta}$ for any $\delta>\frac{1}{2m-1}$, where $m=2,3,\cdots$. The question whether one can improve the Lipschitz regularity in Theorem \ref{coro} to $C^1$ remains for further study.
\end{rema}

Interior $C^2$ estimate for equation (\ref{SL-original}) with $|\Theta|\geq (n-2)\frac{\pi}{2}$ was obtained via works of Warren and Yuan \cite{WY09, WY10} and Wang and Yuan \cite{WdY14}. Interior gradient estimate for equation (\ref{SL-original}) with $|\Theta|\geq (n-2)\frac{\pi}{2}$ was obtained by  Warren and Yuan \cite{WY10}. Interior $C^2$ and gradient estimates for equation (\ref{SL}) in the supercritical phase case $|\theta(x)|> (n-2)\frac{\pi}{2}$ were obtained by Bhattacharya \cite{B1}.  Regularity of convex solution to (\ref{SL-original}) was obtained by Chen, Warren and Yuan \cite{CWY} and by Chen, Shankar and Yuan \cite{CSY} for viscosity solution. Regularity of convex viscosity solution to (\ref{SL}) was obtained by Bhattacharya and Shankar \cite{BS1,BS2}.

To prove the interior $C^2$ estimate, we follow the strategy developed in \cite{WY09,WY10,WdY14}. A key ingredient in our proof is a Jacobi inequality for the largest eigenvalue $\lambda_1$. It was derived via careful study of third order terms. Apart from lack of concavity as mentioned above, we also encounter two new difficulties compared to supercritical phase case and constant phase case. In the supercritical phase case $\theta(x)> (n-2)\frac{\pi}{2}$, we can easily obtain that the smallest eigenvalue $\lambda_n>-C$. This property played an essential role in \cite{B1}. Such property is not available in the critical and supercritical phase case. In principle, the smallest eigenvalue $\lambda_n$ may tends to $-\infty$ while all other eigenvalues $\lambda_i$ tends to $\infty$. Consequently, all coefficients of the linearized operator may be highly degenerate. Compared to (\ref{SL-original}),  we have extra terms involving $\theta(x)$ when we differentiate the equation (\ref{SL}). To overcome these difficulties, we make full use of the Lagrangian phase operator and treat these extra terms with great delicacy. Another feature of our proof is the use of viscosity solution in the case of multiplicity of largest eigenvalue $\lambda_1$, following the work of Brendle, Choi and Daskalopoulos \cite{BCD}. This provides a much more elegant way to derive the inequality.

The organization of the paper is as follows. In section 2, we collect some basic properties relating to Lagrangian submanifold and Lagrangian phase operator. In section 3, we derive the crucial interior $C^2$ estimate and prove Theorem \ref{thm-C^2-intro}. In section 3, we derive the interior to boundary gradient estimate. In section 4, we complete the full a priori estimates and prove Theorem \ref{theorem} and Theorem \ref{coro}.

\section{Preliminaries}

We first recall some background for Lagrangian phase equation, see \cite{B1, HL} for details.

Let $X=(x,Du(x))$ be a Lagrangian submanifold in $\mathbb{R}^n\times\mathbb{R}^n$ for $n\geq 3$, then the metric is given by
\begin{align*}
g=I_n+(D^2u)^2.
\end{align*}

The volume form, gradient and inner product with respect to the metric $g$ are given by
\begin{align*}
dv_g=\sqrt{\det g}dx,\quad &\nabla_gv=\sum_{i,j} g^{ij}v_iX_j,\\
\left\langle \nabla_gv,\nabla_gw\right\rangle_g&=\sum_{i,j} g^{ij}v_iw_j.
\end{align*}

Let $\theta\in (-\frac{n}{2}\pi,\frac{n}{2}\pi)$ be the phase function, suppose $D^2u$ satisfies
\begin{align}\label{SL-Pre}
F(D^2u)=\sum_{i=1}^n \arctan(\lambda_i)=\theta,
\end{align}
where $\lambda_i$'s are the eigenvalues of $D^2u$.

Then the Laplace-Beltrami operator of the metric $g$ is given by
\begin{align*}
\Delta_g=\sum_{i,j} g^{ij}\partial_{ij}-\sum_{j,p,q} g^{jp}\theta_q u_{pq}\partial_j.
\end{align*}

The mean curvature vector is given by
\begin{align*}
H=J\nabla_g\theta,
\end{align*}
where $J$ is the complex structure or the $\frac{\pi}{2}$ rotation matrix in $\mathbb{R}^n\times\mathbb{R}^n$.

\medskip

We now collect some basic properties of Lagrange phase operator.
\begin{lemm}\label{F}
Let $F$ be the Lagrangian phase operator, then we have
\begin{align*}
F^{ij}&=\frac{\partial F}{\partial u_{ij}}=\begin{cases}
\frac{1}{1+\lambda_i^2},\quad &i=j,\\
0, &i\neq j,
\end{cases}\\
F^{ij,kl}&=\frac{\partial^2F}{\partial u_{ij}\partial u_{kl}}=\begin{cases}
-\frac{2\lambda_i}{(1+\lambda_i^2)^2},\quad &i=j=k=l,\\
-\frac{\lambda_i+\lambda_j}{(1+\lambda_i^2)(1+\lambda_j^2)}, & i=l,j=k,\\
0, & \textit{otherwise}.
\end{cases}
\end{align*}
In particular,
\begin{align*}
F^{ij,ji}=\frac{F^{ii}-F^{jj}}{\lambda_i-\lambda_j},\quad i\neq j,\quad \lambda_i\neq \lambda_j.
\end{align*}
\end{lemm}

For Lagrangian phase equation with critical and supercritical phase, we have the following lemma, see \cite{WdY14,Y06} for details.
\begin{lemm}\label{sigma_{n-1}}
Suppose $\lambda_1\geq \lambda_2\geq \cdots\geq \lambda_n$ satisfying (\ref{SL-Pre}) with $\theta\geq \frac{(n-2)\pi}{2}$. Then
\begin{enumerate}
\item $\lambda_1\geq \lambda_2\geq \cdots\geq \lambda_{n-1}>0$, $\lambda_{n-1}\geq |\lambda_n|$.
\item $\lambda_1+(n-1)\lambda_n\geq 0$.
\item $\sigma_k(\lambda_1,\cdots,\lambda_n)\geq 0$ for all $1\leq k\leq n-1$.
\item For any constant $c$ with $|c|\geq \frac{(n-2)\pi}{2}$, the level set $\Gamma=\{\lambda\in \mathbb{R}^n:\sum_{i=1}^n \arctan (\lambda_i)=c\}$
is convex.
\end{enumerate}
\end{lemm}

\section{Interior $C^2$ estimate}

In this section, we will prove the interior $C^2$ estimate for Lagrange phase equation with critical and supercritical phase. We will prove the following theorem which is a litter bit stronger than Theorem \ref{thm-C^2-intro}.

\begin{theo}\label{thm-C^2}
Let $B_2$ be a ball of radius $2$ in $\mathbb{R}^n$ for $n\geq 3$. Let $\theta\in C^{1,1}(B_2\times \mathbb{R})$ with $|\theta(x,u)|\in \big[ \frac{(n-2)\pi}{2}, \frac{n\pi}{2}\big)$ and let $u\in C^4(B_2)$ be a solution of
\begin{align}\label{SL-C^2}
F(D^2u)=\theta(x,u).
\end{align}
Then we have
\begin{align*}
|D^2u(0)|\leq C,
\end{align*}
where $C$ depends only on $n,\|u\|_{C^{0,1}(\overline{B}_1)}$ and $\|\theta\|_{C^{1,1}(\overline{B}_1\times [-M,M])}$. Here $M$ is a universal constant satisfying $\|u\|_{L^\infty(\overline{B}_1)}\leq M$.
\end{theo}

The following is the Jacobi inequality, which is the key for the interior $C^2$ estimate.
\begin{lemm}\label{max}
Let $B_2$ be a ball of radius $2$ in $\mathbb{R}^n$ for $n\geq 3$. Let $\theta\in C^{1,1}(B_2\times \mathbb{R})$ with $\theta(x,u)\in \big[ \frac{(n-2)\pi}{2}, \frac{n\pi}{2}\big)$ and let $u\in C^4(B_2)$ be a solution of (\ref{SL-C^2}). For any $x_0\in B_1$, suppose that $D^2u$ is diagonalized at $x_0$ such that $\lambda_1\geq  \cdots\geq \lambda_n$ and $\lambda_1\geq \Lambda(n)$, where $\Lambda(n)$ is a large constant depending only on $n$. Set $b=\ln\lambda_1$. Then at $x_0$, we have
\begin{align*}
\sum_i F^{ii}b_{ii}\geq c(n)\sum_i F^{ii}b_i^2-C,
\end{align*}
in the viscosity sense, where $c(n)$ depends only on $n$ and $C$ depends only on $n,\|u\|_{C^{0,1}(\overline{B}_1)}$ and $\|\theta\|_{C^{1,1}(\overline{B}_1\times [-M,M])}$. Here $M$ is a universal constant satisfying $\|u\|_{L^\infty(\overline{B}_1)}\leq M$.
\end{lemm}

\begin{proof}
Without loss of generality, we may assume $\lambda_1$ has multiplicity $m$, i.e.
\begin{align*}
\lambda_1=\cdots=\lambda_m>\lambda_{m+1}\geq \cdots\geq \lambda_n.
\end{align*}

By \cite{BCD}, we have
\begin{align*}
{\lambda_1}_i=u_{11i}=u_{lli},\quad 1<l\leq m,
\end{align*}
and
\begin{align*}
{\lambda_1}_{ii}\geq u_{11ii}+2\sum_{p>m}\frac{u_{1pi}^2}{\lambda_1 -\lambda_p},
\end{align*}
in the viscosity sense.

It follows that
\begin{align*}
b_{ii}=\frac{{\lambda_1}_{ii}}{\lambda_1}-\frac{{\lambda_1}_i^2}{\lambda_1^2}\geq \frac{u_{11ii}}{ \lambda_1}+2\sum_{p>m}\frac{u_{1pi}^2}{\lambda_1 (\lambda_1 -\lambda_p )}-\frac{ u_{11i}^2}{\lambda_{1}^2}.
\end{align*}

Thus
\begin{align*}
\sum_i F^{ii}b_{ii}\geq \sum_i F^{ii}\frac{u_{11ii}}{ \lambda_1}+2\sum_i\sum_{p>m}F^{ii}\frac{u_{1pi}^2}{\lambda_1 (\lambda_1 -\lambda_p )}-\sum_i F^{ii}\frac{ u_{11i}^2}{\lambda_{1}^2}.
\end{align*}

Differentiating (\ref{SL-C^2}), we have
\begin{align*}
\sum_i F^{ii}u_{ii11}+\sum_{p,q,r,s} F^{pq,rs}u_{pq1}u_{rs1}\geq -Cu_{11}-C,
\end{align*}
where $C$ depends only on $n,\|u\|_{C^{0,1}(\overline{B}_1)}$ and $\|\theta\|_{C^{1,1}(\overline{B}_1\times [-M,M])}$. Here $M$ is a universal constant satisfying $\|u\|_{L^\infty(\overline{B}_1)}\leq M$.

In the following, we will denote $C$ a universal constant depending only on $n,\|u\|_{C^{0,1}(\overline{B}_1)}$ and $\|\theta\|_{C^{1,1}(\overline{B}_1\times [-M,M])}$, it may change from line to line.

It follows that 
\begin{align}\label{m-1}
\sum_i F^{ii}b_{ii}\geq &  -\sum_{p,q,r,s} \frac{F^{pq,rs}u_{pq1}u_{rs1}}{\lambda_1 }+2\sum_i\sum_{p>m}F^{ii}\frac{u_{1pi}^2}{\lambda_1 (\lambda_1 -\lambda_p )}\\\nonumber
&-\sum_i F^{ii}\frac{ u_{11i}^2}{\lambda_{1}^2}-C.
\end{align}

By Lemma \ref{F}, we have
\begin{align*}
-\sum_{p,q,r,s} F^{pq,rs}u_{pq1}u_{rs1}=&-\sum_i F^{ii,ii}u_{ii1}^2-\sum_{i\neq j}F^{ij,ji}u_{ij1}^2\\
=&\ 2\sum_i \lambda_i(F^{ii})^2 u_{ii1}^2+\sum_{i\neq j}\frac{\lambda_i+\lambda_j}{(1+\lambda_i^2)(1+\lambda_j^2)} u_{ij1}^2.
\end{align*}

Plugging into (\ref{m-1}), we have
\begin{align}\label{m-2}
\sum_i F^{ii}b_{ii}\geq&\  2\sum_i \frac{\lambda_i}{\lambda_1} (F^{ii})^2 u_{ii1}^2+\sum_{i\neq j}\frac{\lambda_i+\lambda_j}{\lambda_1(1+\lambda_i^2)(1+\lambda_j^2)} u_{ij1}^2\\\nonumber
&+2\sum_i\sum_{p>m}F^{ii}\frac{u_{1pi}^2}{\lambda_1 (\lambda_1 -\lambda_p )}-\sum_i F^{ii}\frac{ u_{11i}^2}{\lambda_{1}^2}-C.
\end{align}

Note that
\begin{align*}
2\sum_i\sum_{p>m}F^{ii}\frac{u_{1pi}^2}{\lambda_1 (\lambda_1 -\lambda_p )}&\geq 2\sum_{p>m}F^{pp}\frac{u_{1pp}^2}{\lambda_1 (\lambda_1 -\lambda_p )}+ 2\sum_{p>m}F^{11}\frac{u_{1p1}^2}{\lambda_1 (\lambda_1 -\lambda_p )}\\
&=2\sum_{i>m}F^{ii}\frac{u_{ii1}^2}{\lambda_1 (\lambda_1 -\lambda_i )}+2\sum_{i>m}F^{11}\frac{u_{11i}^2}{\lambda_1 (\lambda_1 -\lambda_i )}.
\end{align*}

By Lemma \ref{sigma_{n-1}}, we have $\lambda_i+\lambda_j\geq 0$ for all $i\neq j$, together with Lemma \ref{F}, we have
\begin{align*}
&\sum_{i\neq j}\frac{\lambda_i+\lambda_j}{\lambda_1 (1+\lambda_i^2)(1+\lambda_j^2)} u_{ij1}^2\\
\geq&\ 2\sum_{1<i\leq m}  \frac{\lambda_1+\lambda_i}{\lambda_1(1+\lambda_1^2)(1+\lambda_i^2)} u_{11i}^2+ 2 \sum_{i>m} \frac{\lambda_1+\lambda_i}{\lambda_1(1+\lambda_1^2)(1+\lambda_i^2)}   u_{11i}^2\\
=&\ 4\sum_{1<i\leq m} (F^{ii})^2 u_{11i}^2+  2\sum_{i>m} \frac{F^{ii}-F^{11}}{\lambda_1 (\lambda_1 -\lambda_i )}u_{11i}^2
\end{align*}

Combining the above two inequalities and plugging into (\ref{m-2}), we have
\begin{align*}
\sum_i F^{ii}b_{ii}\geq &\  2\sum_i\frac{\lambda_i}{\lambda_1} (F^{ii})^2 u_{ii1}^2+2\sum_{i>m}F^{ii}\frac{u_{ii1}^2}{\lambda_1 (\lambda_1 -\lambda_i )}+4\sum_{1<i\leq m} (F^{ii})^2 u_{11i}^2\\
&+2\sum_{i>m}F^{ii}\frac{u_{11i}^2}{\lambda_1 (\lambda_1 -\lambda_i )}-\sum_i F^{ii}\frac{ u_{11i}^2}{\lambda_{1}^2}-C.
\end{align*}

The coefficient for $u_{111}^2$ is
\begin{align*}
 2 (F^{11})^2 -\frac{F^{11}}{\lambda_1^2}=F^{11}\left(\frac{2}{1+\lambda_1^2}  -\frac{1}{\lambda_1^2}\right)=\frac{\lambda_1^2-1}{\lambda_1^2}(F^{11})^2.
\end{align*}

The coefficient for $u_{11i}^2$ for $1<i\leq m$ is
\begin{align*}
4(F^{ii})^2-\frac{ F^{ii}}{\lambda_{1}^2}=F^{ii} \left(4F^{ii}-\frac{1}{\lambda_1^2}\right)=\frac{3\lambda_1^2-1}{\lambda_1^2}(F^{ii})^2.
\end{align*}

The coefficient for $u_{11i}^2$ for $i>m$ is
\begin{align*}
2\frac{F^{ii}}{\lambda_1 (\lambda_1 -\lambda_i )}-\frac{ F^{ii}}{\lambda_{1}^2}=\frac{\lambda_1+\lambda_i}{\lambda_1^2 (\lambda_1 -\lambda_i )}F^{ii} .
\end{align*}

The coefficient for $u_{ii1}^2$ for $1<i\leq m$ is $2(F^{ii})^2$.

The coefficient for $u_{ii1}^2$ for $i>m$ is
\begin{align*}
2\frac{\lambda_i}{\lambda_1} (F^{ii})^2 +2\frac{F^{ii}}{\lambda_1 (\lambda_1 -\lambda_i )}=2 \frac{F^{ii}}{\lambda_1} \left(F^{ii}\lambda_i+\frac{1}{\lambda_1 -\lambda_i } \right)=2\frac{1+\lambda_1\lambda_i}{\lambda_1(\lambda_1-\lambda_i)}(F^{ii})^2.
\end{align*}

It follows that
\begin{align*}
\sum_i F^{ii}b_{ii}\geq  &\ \frac{\lambda_1^2-1}{\lambda_1^2}(F^{11})^2u_{111}^2+\sum_{1<i\leq m}\frac{3\lambda_1^2-1}{\lambda_1^2}(F^{ii})^2u_{11i}^2+\sum_{i>m} \frac{\lambda_1+\lambda_i}{\lambda_1^2 (\lambda_1 -\lambda_i )}F^{ii} u_{11i}^2\\\nonumber
&+2\sum_{1<i\leq m}(F^{ii})^2u_{ii1}^2+2\sum_{i>m} \frac{1+\lambda_1\lambda_i}{\lambda_1(\lambda_1-\lambda_i)}(F^{ii})^2 u_{ii1}^2-C
\end{align*}

By Lemma \ref{sigma_{n-1}} and the fact $u_{111}=u_{ii1}$ for $1<i\leq m$, we have
\begin{align*}
\sum_i F^{ii}b_{ii}\geq &\ \frac{1}{m}\sum_{i\leq m} \left(2m-1-\frac{1}{\lambda_1^2}\right)  (F^{ii})^2u_{ii1}^2+\sum_{1<i\leq m}c(n)(F^{ii})^2u_{11i}^2\\
&+\sum_{i> m} c(n)F^{ii}\frac{u_{11i}^2}{\lambda_1^2}+\sum_{i>m} \frac{2\lambda_i}{\lambda_1-\lambda_i}(F^{ii})^2 u_{ii1}^2-C,
\end{align*}

Choosing $\epsilon(n)=\Lambda(n)^{-2}$ to be sufficiently small, we have
\begin{align}\label{m-3}
\sum_i F^{ii}b_{ii}\geq &\ \sum_{i\leq m} \frac{2m-1-\epsilon}{m}  (F^{ii})^2u_{ii1}^2+\sum_{i>1}c(n) F^{ii}\frac{u_{11i}^2}{\lambda_1^2}\\\nonumber
&+\sum_{i>m} \frac{2\lambda_i}{\lambda_1-\lambda_i}(F^{ii})^2 u_{ii1}^2-C,
\end{align}

Note that $\lambda_i\geq 0$ for $1\leq i\leq n-1$. If $\lambda_n\geq 0$, then we have proved the lemma.

\medskip

From now on, we assume $\lambda_n<0$. It follows that $m\leq n-1$.

For simplicity, denote $F^{ii}u_{ii1}=t_i$. Differentiating (\ref{SL-C^2}), we have
\begin{align*}
\sum_i F^{ii}u_{ii1}=\sum_{i} t_i=\theta_1+\theta_uu_1.
\end{align*}

It follows that
\begin{align*}
t_n^2= \left( \sum_{i\neq n}t_i-\theta_1-\theta_uu_1\right) ^2\leq \left( \sum_{i\neq n}t_i\right) ^2+C\sum_{i\neq n}|t_i|+C.
\end{align*}

By Cauchy-Schwarz inequality, we have
\begin{align*}
\left( \sum_{i\neq n}t_i\right) ^2\leq \left( \sum_{i\leq m} \frac{2m-1-2\epsilon}{m} t_i^2+\sum_{m<i<n}\frac{2\lambda_i}{\lambda_1 -\lambda_i} t_i^2\right)\left(\sum_{i\leq m}\frac{m}{2m-1-2\epsilon} +\sum_{m<i<n}\frac{\lambda_1- \lambda_i}{2\lambda_i}\right).
\end{align*}

Plugging into (\ref{m-3}), we have
\begin{align*}
&\sum_i F^{ii}b_{ii} \\
\geq &\left( \sum_{i\leq m} \frac{2m-1-2\epsilon}{m} t_i^2+\sum_{m<i<n}\frac{2\lambda_i}{\lambda_1-\lambda_i}t_i^2 \right)\\
&\cdot\left(1+\frac{2\lambda_n}{\lambda_1-\lambda_n}\left(\sum_{i\leq m}\frac{m}{2m-1-2\epsilon} +\sum_{m<i<n}\frac{\lambda_1-\lambda_i}{2\lambda_i} \right)\right)\\
& +\sum_{i\leq m} \frac{\epsilon}{m} (F^{ii})^2u_{ii1}^2+\sum_{i>1}c(n)F^{ii}\frac{u_{11i}^2}{\lambda_1^2}  -C\frac{|\lambda_n|}{\lambda_1-\lambda_n} \sum_{i\neq n}|t_i|-C.
\end{align*}

The second bracket equals
\begin{align*}
&\frac{2\lambda_n}{\lambda_1-\lambda_n} \left( \frac{\lambda_1-\lambda_n}{2\lambda_n}+\sum_{i\leq m}\frac{m}{2m-1-2\epsilon}+\sum_{m<i<n}\frac{\lambda_1-\lambda_i}{2\lambda_i}  \right)\\
=&\frac{2\lambda_n}{\lambda_1-\lambda_n} \left(\frac{m^2}{2m-1-2\epsilon}+\sum_{i} \frac{\lambda_1-\lambda_i}{2\lambda_i}  \right)\\
= &\frac{2\lambda_n}{\lambda_1-\lambda_n} \left( \frac{\lambda_1}{2} \frac{\sigma_{n-1}}{\sigma_n}+ \frac{m^2}{2m-1-2\epsilon}-\frac{n}{2}  \right)\\
\geq &\frac{2\lambda_n}{\lambda_1-\lambda_n} \left(\frac{m^2}{2m-1-2\epsilon}-\frac{n}{2} \right)\geq \frac{|\lambda_n|}{\lambda_1-\lambda_n}  c(n),
\end{align*}
where we have used the fact that $\sigma_{n-1}\geq 0$, $\sigma_n<0$, $m\leq n-1$ and we choose $\epsilon$ to be a small constant depending only on $n$.

Thus
\begin{align*}
\sum_i F^{ii}b_{ii}\geq  & \left(\sum_{i\leq m}\frac{2m-1-2\epsilon}{m}t_i^2+\sum_{m<i<n}\frac{2\lambda_i}{\lambda_1-\lambda_i}t_i^2 \right)\frac{|\lambda_n|}{\lambda_1-\lambda_n}  c(n)-C\frac{|\lambda_n|}{\lambda_1-\lambda_n} \sum_{i\neq n}|t_i|\\
&+\sum_{i\leq m} \frac{\epsilon}{m} (F^{ii})^2u_{ii1}^2+\sum_{i>1}c(n)F^{ii}\frac{u_{11i}^2}{\lambda_1^2}-C.
\end{align*}

For $i\leq m$, we have
\begin{align*}
\frac{2m-1-2\epsilon}{m} t_i^2\frac{|\lambda_n|}{\lambda_1-\lambda_n}  c(n)-C\frac{|\lambda_n|}{\lambda_1-\lambda_n}|t_i|\geq -C\frac{|\lambda_n|}{\lambda_1-\lambda_n}\geq -C.
\end{align*}

For $m<i<n$, we have
\begin{align*}
\frac{2\lambda_i}{\lambda_1-\lambda_i}t_i^2\frac{|\lambda_n|}{\lambda_1-\lambda_n} c(n)-C\frac{|\lambda_n|}{\lambda_1-\lambda_n}|t_i|\geq  2c(n)\frac{|\lambda_n|^2}{(\lambda_1-\lambda_n)^2}t_i^2-C\frac{|\lambda_n|}{\lambda_1-\lambda_n}|t_i|\geq -C.
\end{align*}

It follows that
\begin{align*}
\sum_i F^{ii}b_{ii} \geq & \sum_{i\leq m} \frac{\epsilon}{m} (F^{ii})^2u_{ii1}^2+\sum_{i>1} c(n)F^{ii}\frac{u_{11i}^2}{\lambda_1^2}-C\\
\geq  &\  \epsilon\frac{\lambda_1^2}{1+\lambda_1^2}F^{11}b_1^2+\sum_{i> 1} c(n)F^{ii} b_i^2-C\\
\geq &\   c(n)\sum_i F^{ii} b_i^2-C.
\end{align*}

The lemma is now proved.

\end{proof}

We can integrate the Jacobi inequality to obtain the following integral Jacobi inequality.
\begin{lemm}\label{integral jacobi}
Let $B_2$ be a ball of radius $2$ in $\mathbb{R}^n$ for $n\geq 3$. Let $\theta\in C^{1,1}(B_2\times \mathbb{R})$ with $\theta(x,u)\in \big[ \frac{(n-2)\pi}{2}, \frac{n\pi}{2}\big)$ and let $u\in C^4(B_2)$ be a solution of (\ref{SL-C^2}).  Let $\lambda_1$ be the largest eigenvalue of $D^2u$. Let
\begin{align*}
b=\ln \max\{\Lambda(n),\lambda_1\},
\end{align*}
where $\Lambda(n)$ is given by Lemma \ref{max}. Then for all nonnegative $\varphi\in C_c^\infty(B_1)$, there holds
\begin{align*}
-\int_{B_1}\left\langle \nabla_g\phi,\nabla_gb\right\rangle_gdv_g\geq c(n)\int_{B_1}\phi|\nabla_gb|^2dv_g-C\int_{B_1}\phi dv_g,
\end{align*}
where $c(n)$ depends only on $n$ and $C$ depends only on $n,\|u\|_{C^{0,1}(\overline{B}_1)}$ and $\|\theta\|_{C^{1,1}(\overline{B}_1\times [-M,M])}$. Here $M$ is a universal constant satisfying $\|u\|_{L^\infty(\overline{B}_1)}\leq M$.
\end{lemm}

\begin{proof}
For any $x_0\in B_1$, without loss of generality, we may assume $D^2u$ is diagonalized at $x_0$. By the definition of $\Delta_g$ and Lemma \ref{max}, for $\lambda_1\geq \Lambda(n)$, we have
\begin{align*}
\Delta_g (\ln \lambda_1)&=\sum_{i,j} g^{ij}(\ln \lambda_1)_{ij}-\sum_{j,p,q} g^{jp} \theta_q u_{pq}(\ln \lambda_1)_j\\
&=\sum_i F^{ii}(\ln \lambda_1)_{ii}-\sum_i F^{ii}\theta_iu_{ii}(\ln \lambda_1)_i\\
&\geq c(n)\sum_i F^{ii}(\ln \lambda_1)_i^2-C\sum_i F^{ii}\left(\epsilon (\ln \lambda_1)_i^2+\frac{1}{\epsilon}\lambda_i^2\right)-C\\
&\geq \frac{c(n)}{2}\sum_i F^{ii} (\ln \lambda_1)_i^2-C=\frac{c(n)}{2}|\nabla_g(\ln \lambda_1)|^2-C,
\end{align*}
in the viscosity sense, where $c(n)$ depends only on $n$ and $C$ depends only on $n,\|u\|_{C^{0,1}(\overline{B}_1)}$ and $\|\theta\|_{C^{1,1}(\overline{B}_1\times [-M,M])}$. 

It is easy to see $b=\ln \max\{\Lambda(n),\lambda_1\}$ satisfies the above inequality in the viscosity sense. By Theorem 1 in \cite{I}, $b$ also satisfies the above inequality in the distribution sense. The lemma now follows by integration by parts.
\end{proof}

\textit{Proof of Theorem \ref{thm-C^2}:} 
\begin{proof}
We first note that we only need to prove the case $\theta(x,u)\geq \frac{(n-2)\pi}{2}$. Once we have the integral Jacobi inequality, the interior $C^2$ estimate follows from the sophisticated theory developed in \cite{WY09,WY10,WdY14}, based on the Michael-Simon type inequality \cite{B1,MS}, see \cite{B1,WdY14} for details.
\end{proof}

\section{Gradient estimate}
\begin{lemm}\label{C^1}
Let $\Omega$ be a bounded domain in $\mathbb{R}^n$ for $n\geq 3$. Let $\varphi\in C^{0,1}(\partial\Omega)$ and let $\theta\in C^{0,1}(\overline{\Omega})$ with $\theta(x)\in \left(-\frac{n\pi}{2}, \frac{n\pi}{2}\right)$. Let $u\in C^3(\Omega)\cap C^{0,1}(\overline{\Omega})$ be a solution of the Dirichlet problem
\begin{align*}
\begin{cases}
F(D^2u)=\theta(x),\quad \textit{in }\Omega,\\
u=\varphi,  \quad\textit{on } \partial\Omega.
\end{cases}
\end{align*}
Then we have 
\begin{align*}
\max_{\overline{\Omega}}|\nabla u|\leq C\left(1+\max_{\partial \Omega}|\nabla u|\right) ,
\end{align*}
where $C$ depends only on $n,\|u\|_{L^\infty(\overline{\Omega})}$ and $\|\theta\|_{C^{0,1}(\overline{\Omega})}$.
\end{lemm}

\begin{proof}
Without loss of generality, we may assume $0<M<u<2M$.

We have
\begin{align*}
|\nabla u|_i=\frac{1}{2} |\nabla u|^{-1}(|\nabla u|^2)_i= |\nabla u|^{-1}\sum_k u_ku_{ki}.
\end{align*}

Thus
\begin{align*}
|\nabla u|_{ii}&=|\nabla u|^{-1}\sum_k (u_ku_{kii}+u_{ki}^2)-|\nabla u|^{-2}|\nabla u|_i\sum_k  u_ku_{ki}\\
&=|\nabla u|^{-1}\sum_k (u_ku_{kii}+u_{ki}^2)-|\nabla u|^{-3}\left(\sum_k  u_ku_{ki}\right)^2.
\end{align*}

Consider the test function
\begin{align*}
|\nabla u|+\epsilon u^2,
\end{align*}
where $\epsilon>0$ is a small constant to be chosen later. 

Assume it attains its maximum at an interior point $x_0\in \Omega$. Without loss of generality, we may assume $|\nabla u|>1$ at $x_0$, for otherwise we have already obtained the desired estimate by the definition of test function. Choose a coordinate system such that $D^2u$ is diagonalized at $x_0$ and $u_n\geq \frac{|\nabla u|}{\sqrt{n} }>0$ at $x_0$.

At $x_0$, we have
\begin{align}\label{gradient-critical}
0= |\nabla u|_i+2\epsilon uu_i= |\nabla u|^{-1}\sum_k u_ku_{ki}+2\epsilon uu_i.
\end{align}
and
\begin{align*}
0&\geq  |\nabla u|^{-1}\sum_k\left(   u_ku_{kii}+u_{ki}^2 \right)-|\nabla u|^{-3}\left(\sum_k  u_ku_{ki}\right)^2+2\epsilon uu_{ii}+2\epsilon u_i^2\\
&=  |\nabla u|^{-1}\sum_k  u_ku_{kii}+|\nabla u|^{-3} u_{ii}^2\left(|\nabla u|^2-u_i^2\right)+2\epsilon uu_{ii}+2\epsilon u_i^2\\
&\geq  |\nabla u|^{-1}\sum_k  u_ku_{kii}+2\epsilon uu_{ii}+2\epsilon u_i^2.
\end{align*}

Contracting with $F^{ii}$, we have
\begin{align*}
0&\geq  |\nabla u|^{-1}\sum_{k,i}  F^{ii}u_ku_{kii}+2\epsilon u\sum_i F^{ii}u_{ii}+2\epsilon\sum_i  F^{ii}u_i^2\\
&\geq   |\nabla u|^{-1}\sum_{k}  u_k \theta_k+2\epsilon u \sum_i \frac{u_{ii}}{1+u_{ii}^2}+ \frac{2\epsilon}{n} F^{nn}|\nabla u|^2\\
&\geq  -C+ \frac{2\epsilon}{n} F^{nn}|\nabla u|^2.
\end{align*}
where $C$ denotes a universal constant depending only on $n,\|u\|_{L^\infty(\overline{\Omega})}$ and $\|\theta\|_{C^{0,1}(\overline{\Omega})}$, it may change from line to line. 

By the critical equation (\ref{gradient-critical}), we have
\begin{align*}
|u_{nn}|=2\epsilon u|\nabla u|.
\end{align*}

We may assume $|\nabla u|$ is large enough, for otherwise we have already obtained the estimate by the definition of the test function. Thus
\begin{align*}
0\geq & -C+\frac{2\epsilon}{n} F^{nn}|\nabla u|^2=-C+\frac{2\epsilon}{n} \frac{|\nabla u|^2}{1+(2\epsilon u |\nabla u|)^2}\geq  -C+\frac{C}{\epsilon}>0,
\end{align*}
by choosing $\epsilon$ small.

Thus the maximum must occur at $\partial\Omega$. The lemma is now proved.

\end{proof}

\begin{rema}
The above lemma works for all phase functions.
\end{rema}

\section{Proof of Theorem \ref{theorem} and Theorem \ref{coro}} 

Let us first recall a lemma concerning the $C^{0,1}$ estimate for special Lagrangian equation with supercritical phase, see (3.4) and (3.9) in \cite{B2}.
\begin{lemm}\label{subsolution}
Let $\Omega$ be a bounded, strictly convex domain in $\mathbb{R}^n$ for $n\geq 3$. Let $\varphi\in C^{1,1}(\partial\Omega)$ and let $\underline{\Theta}\in \left(\frac{(n-2)\pi}{2},\frac{n\pi}{2}\right)$ be any constant. Let $\underline{u}$ be a viscosity solution of the Dirichlet problem
\begin{align*}
\begin{cases}
F(D^2\underline{u})=\underline{\Theta},\quad \textit{in }\Omega,\\
\underline{u}=\varphi,  \quad\textit{on } \partial\Omega.
\end{cases}
\end{align*}
Then we have
\begin{align*}
\|\underline{u}\|_{C^{0,1}\overline{\Omega})}\leq C,
\end{align*}
where $C$ depends only on $n,\Omega$ and $\|\varphi\|_{C^{1,1}(\partial\Omega)}$.
\end{lemm}

\begin{proof}
Note that we only need comparison principle for viscosity solution of equation (\ref{SL-original}) to establish (3.4) and (3.9) in \cite{B2}. Since $\underline{\Theta}$ is a constant, such comparison principle is available (see for instance Theorem 6.1 in \cite{B2}).
\end{proof}

Similarly, we have
\begin{lemm}\label{supsolution}
Let $\Omega$ be a bounded, strictly convex domain in $\mathbb{R}^n$ for $n\geq 3$. Let $\varphi\in C^{1,1}(\partial\Omega)$ and let $\overline{\Theta}\in \left(-\frac{n\pi}{2},-\frac{(n-2)\pi}{2}\right)$ be any constant. Let $\overline{u}$ be a viscosity solution of the Dirichlet problem
\begin{align*}
\begin{cases}
F(D^2\overline{u})=\overline{\Theta},\quad \textit{in }\Omega,\\
\overline{u}=\varphi,  \quad\textit{on } \partial\Omega.
\end{cases}
\end{align*}
Then we have
\begin{align*}
\|\overline{u}\|_{C^{0,1}(\overline{\Omega})}\leq C,
\end{align*}
where $C$ depends only on $n,\Omega$ and $\|\varphi\|_{C^{1,1}(\partial\Omega)}$.
\end{lemm}

\textit{Proof of Theorem \ref{theorem}:}

\begin{proof}

We first note that we only need to prove the theorem for $\theta(x)\geq \frac{(n-2)\pi}{2}$. In fact, for $\theta(x)\leq - \frac{(n-2)\pi}{2}$, let $v$ be the solution of
\begin{align*}
\begin{cases}
F(D^2v)=-\theta(x),\quad \textit{in }\Omega,\\
v=-\varphi,  \quad\textit{on } \partial\Omega.
\end{cases}
\end{align*}

Then $-v$ is the desired solution.

\medskip

Since $|\theta(x)|<\frac{n\pi}{2}$, we may choose $-\frac{n\pi}{2}<\overline{\Theta}\leq \theta(x)\leq  \underline{\Theta}<\frac{n\pi}{2}$ such that $\underline{\Theta}\in \left(\frac{(n-2)\pi}{2},\frac{n\pi}{2}\right)$ and $\overline{\Theta}\in \left(-\frac{n\pi}{2},-\frac{(n-2)\pi}{2}\right)$. Let $\underline{u}$ and $\overline{u}$ be the solution in Lemma \ref{subsolution} and Lemma \ref{supsolution} respectively.

By comparison principle for viscosity solution of equation (\ref{SL-original}), we have
\begin{align*}
\underline{u}\leq u \leq \overline{u}.
\end{align*}

Since $u=\underline{u}=\overline{u}=\varphi$ on $\partial\Omega$, the gradient of $u$ on $\partial\Omega$ is bounded by the gradient of $\underline{u}$ and the gradient of $\overline{u}$ on $\partial\Omega$.

By Lemma \ref{C^1}, Lemma \ref{subsolution} and Lemma \ref{supsolution}, we have
\begin{align*}
\|u\|_{C^{0,1}(\overline{\Omega})}\leq C,
\end{align*}
where $C$ depends only on $n,\Omega, \|\varphi\|_{C^{1,1}(\partial\Omega)}$ and $\|\theta\|_{C^{0,1}(\overline{\Omega})}$.

By Theorem \ref{thm-C^2}, 
\begin{align*}
\|u\|_{C^{1,1}_{loc}(\Omega)}\leq C,
\end{align*}
where $C$ depends only on $n,\Omega, \|\varphi\|_{C^{1,1}(\partial\Omega)}$ and $\|\theta\|_{C^{1,1}(\overline{\Omega})}$. Here by $\|u\|_{C^{1,1}_{loc}(\Omega)}\leq C$, we mean for any open set $\Omega^\prime$ with $\Omega^\prime\subset \overline{\Omega^\prime}\subset \Omega$, we have $\|u\|_{C^{1,1}(\overline{\Omega^\prime})}\leq C$, where $C$ depends in addition on $\Omega^\prime$. The definition of $C^{2,\alpha}_{loc}(\Omega)$ and $C^{3,\alpha}_{loc}(\Omega)$ below are analogous.

Although $F$ is not concave, by Lemma \ref{sigma_{n-1}}, the level set of $F$ with critical and supercritical phase is in fact convex. This allows us to perform the extended version of Evans-Krylov $C^{2,\alpha}$ estimate \cite{E,Kry83}, see \cite{CJY} for details. Thus
\begin{align*}
\|u\|_{C^{2,\alpha}_{loc}(\Omega)}\leq C,
\end{align*}
where $C$ depends only on $n,\Omega, \|\varphi\|_{C^{1,1}(\partial\Omega)}$ and $\|\theta\|_{C^{1,1}(\overline{\Omega})}$.

By Schauder estimate, we have
\begin{align*}
\|u\|_{C^{3,\alpha}_{loc}(\Omega)}\leq C,
\end{align*}
where $C$ depends only on $n,\Omega, \|\varphi\|_{C^{1,1}(\partial\Omega)}$ and $\|\theta\|_{C^{1,1}(\overline{\Omega})}$.

Now let $\theta_k(x)\in (\frac{(n-2)\pi}{2},\frac{n\pi}{2})$ be a sequence of functions converging to $\theta(x)$, by \cite{B2}, for each $k$, there exists a unique solution $u_k$ solving the Dirichlet problem
\begin{align*}
\begin{cases}
F(D^2u_k)=\theta_k(x),\quad \textit{in }\Omega,\\
u_k=\varphi, \quad \textit{on }\partial\Omega.
\end{cases}
\end{align*}

By estimate above, we can extract a converging subsequence $u_k\rightarrow u$ in $C^{3,\alpha}_{loc}(\Omega)\cap C^{0,1}(\overline{\Omega})$. The existence part now follows. The uniqueness follows from maximum principle for fully nonlinear elliptic equation. The theorem is now proved.

\end{proof}

\textit{Proof of Theorem \ref{coro}:}

\begin{proof}
We only need to prove for $-\frac{(n-2)\pi}{2}< \Theta<\frac{(n-2)\pi}{2}$. 

For $\epsilon>0$ and each $1\leq j\leq n$, define 
\begin{align*}
u_{\epsilon,j}=u(x+\epsilon e_j),\quad \Omega_{\epsilon,j}=\{x\in \Omega:x+\epsilon e_j\in \Omega\}.
\end{align*}

We will first show that
\begin{align*}
\max_{\overline{\Omega_{\epsilon,j}}}(u_{\epsilon,j}-u)\leq \max_{\partial\Omega_{\epsilon,j}}(u_{\epsilon,j}-u).
\end{align*}

In fact, let
\begin{align*}
\max_{\partial\Omega_{\epsilon,j}}(u_{\epsilon,j}-u)=M.
\end{align*}

Then we have 
\begin{align*}
u_{\epsilon,j}\leq u+M ,\quad \textit{on } \partial\Omega_{\epsilon,j}.
\end{align*}

Note that both $u_{\epsilon,j}$ and $u+M$ are viscosity solutions of equation (\ref{SL-original}), by comparison principle for viscosity solution of equation (\ref{SL-original}), we have 
\begin{align*}
u_{\epsilon,j}\leq u+M ,\quad \textit{in } \Omega_{\epsilon,j}.
\end{align*}

It follows that
\begin{align*}
\max_{\overline{\Omega_{\epsilon,j}}}(u_{\epsilon,j}-u)\leq M=\max_{\partial\Omega_{\epsilon,j}}(u_{\epsilon,j}-u).
\end{align*}

Similarly, we have
\begin{align*}
\max_{\overline{\Omega_{\epsilon,j}}}(u-u_{\epsilon,j})\leq \max_{\partial\Omega_{\epsilon,j}}(u-u_{\epsilon,j}).
\end{align*}

Consequently,
\begin{align*}
\max_{\overline{\Omega_{\epsilon,j}}}|u_{\epsilon,j}-u|\leq \max_{\partial\Omega_{\epsilon,j}}|u_{\epsilon,j}-u|.
\end{align*}

Therefore,
\begin{align*}
\max_{\overline{\Omega_{\epsilon,j}}}\frac{|u(x+\epsilon e_j)-u(x)|}{\epsilon} \leq \max_{\partial\Omega_{\epsilon,j}}\frac{|u(x+\epsilon e_j)-u(x)|}{\epsilon} .
\end{align*}

Let $\epsilon\rightarrow 0$, then $\Omega_{\epsilon,j}\rightarrow \Omega$ and $\partial\Omega_{\epsilon,j}\rightarrow\partial\Omega$, it follows that
\begin{align*}
\max_{\overline{\Omega}}|\nabla u|\leq \max_{\partial\Omega}|\nabla u|.
\end{align*}

Now let $-\frac{n\pi}{2}<\overline{\Theta}<-\frac{(n-2)\pi}{2}< \Theta<\frac{(n-2)\pi}{2}<\underline{\Theta}<\frac{n\pi}{2}$ and let $\underline{u}$ and $\overline{u}$ be the solution in Lemma \ref{subsolution} and Lemma \ref{supsolution} respectively.

By comparison principle for viscosity solution of equation (\ref{SL-original}), we have
\begin{align*}
\underline{u}\leq u \leq \overline{u}.
\end{align*}

Since $u=\underline{u}=\overline{u}=\varphi$ on $\partial\Omega$, the gradient of $u$ on $\partial\Omega$ is bounded by the gradient of $\underline{u}$ and the gradient of $\overline{u}$ on $\partial\Omega$.

By Lemma \ref{subsolution} and Lemma \ref{supsolution}, we have
\begin{align*}
\|u\|_{C^{0,1}(\overline{\Omega})}\leq C,
\end{align*}
where $C$ depends only on $n,\Omega$ and $\|\varphi\|_{C^{1,1}(\partial\Omega)}$.

The Theorem is now proved.
\end{proof}

\noindent

{\it Acknowledgement}: The author would like to thank Xiangwen Zhang for enlightening conversations. The author would also like to thank Professors Pengfei Guan, YanYan Li and Yu Yuan for their interests and comments. We would also like to thank the anonymous referees for insightful suggestions, which helps to improve the exposition of the paper. In particular, we would like to thank the anonymous referee for pointing out Theorem \ref{coro} as well as its proof, which were previously not known to us.

\end{document}